\documentclass[12pt]{amsart}
\usepackage[margin=1in]{geometry}
\usepackage{latexsym}
\usepackage{float}
\usepackage{amssymb}
\usepackage[all]{xy}
\usepackage{epsfig}
\usepackage{color}
\usepackage{graphics}
\usepackage{hyperref}

\newtheorem{theorem}{Theorem}[section]

\newtheorem{Thm}[theorem]{Theorem}
\newtheorem{Lem}[theorem]{Lemma}
\newtheorem{Prop}[theorem]{Proposition}
\newtheorem{Cor}[theorem]{Corollary}

\begin{document}

\title{ Rectangle condition and its applications}
\author{Bo-hyun Kwon}
\address{Korea University, Department of Mathematics, Seoul, Republic of Korea}\email{bortire74@gmail.com}
\date{April 5,  2017}
\maketitle

\begin{abstract}
In this paper, we define the rectangle condition on the bridge sphere  for a $n$-bridge decomposition of a knot whose definition is analogous to the definition of the rectangle condition for Heegaard splittings of $3$-manifolds. We show that the satisfaction of the rectangle condition for a $n$-bridge decomposition can guarantee that the Hempel distance for the $n$-bridge decomposition is greater than or equal to $2$. In particular, we give an interesting family of alternating 3-bridge knots by using the rectangle condition and a modified train track argument.

\end{abstract}

\section{Introduction}\label{B1}

 Casson and Gordon~\cite{1} introduced the $\emph{rectangle condition}$ on   Heegaard surfaces to show strong irreducibility of Heegaard splittings of $3$-manifolds. A Heegaard splitting $V\cup_SW$ is $\emph{strongly irreducible}$ if for any pair of essential disks $D_1(\subset V)$ and $D_2(\subset W)$, $\partial D_1$ meets $\partial D_2$ in $S=V\cap W$. In other words, the Hempel distance of the Heegaard splitting is greater than or equal to two.
 We consider the Heegaard surface $S$ as the branched double covering of a $2n$-punctured sphere, denoted by $\Sigma_{0,2n}$. Then the natural question is if there is  a similar criterion  for $n$-bridge decompositions of a knot to check whether or not the Hempel distance is greater than or equal to two. Currently, K. Takao~\cite{10} defined the $\emph{well-mixed condition}$ for a $n$-bridge decomposition of a knot which is a variation of the rectangle condition  and the satisfaction of the condition  guarantees the Hempel distance to be at least two. In this paper, we  define the  \emph{rectangle condition} on the bridge sphere for a $n$-bridge decomposition of a knot based on the classical definition and show that the Hempel distance is greater than or equal to two if two pants decompositions related to the $n$-bridge decomposition satisfy the rectangle condition.  To check the satisfaction of the well-mixed condition,  two pants decompositions obtained from the definition need adjacent subarcs  in a conditional  diagram on $\Sigma_{0,2n}$. The adjacent subarcs come from different simple arcs connecting two punctures in $\Sigma_{0,2n}$.  The adjacent subarcs are possibly not parallel sides of a rectangle.  Actually, the well-mixed condition needs more conditions to be checked than the conditions for the rectangle condition we  define. For example,  the well-mixed condition needs 18 conditions to be checked  for a $3$-bridge decomposition of a knot but the rectangle condition needs 9 conditions(rectangles)  for one. So, the obvious advantage of the rectangle condition  would be less conditions to be checked. However, the parallel sides of each rectangle for the rectangle condition are usually longer than the adjacent arcs for well-mixed condition in $\Sigma_{0,2n}$. In other words, there are more obstacles to have rectangles than adjacent arcs.  So, it is hard to say that the rectangle condition is a better criterion than the well-mixed condition.  Moreover, we use the same technique to prove  the  main result, Proposition~\ref{T3}.    In order to give another effectiveness of the rectangle condition  I would give an algorithm to check whether or not two pants decompositions on the bridge sphere for a given $n$-bridge decomposition of a knot satisfy the rectangle condition in Section~\ref{B5}. The figure~\ref{A4} is a diagram that satisfies the rectangle condition but not the well-mixed condition.\\
 
  Hempel distance is a measurement of complexity for Heegaard splittings of $3$-manifolds~\cite{4}. Bachman and Schleimer~\cite{0} transfered the definition from Heegaard surfaces to $n$-bridge spheres to compute the complexity of $n$-bridge decompositions. Actually, Hempel distance for $n$-bridge decompositions of  knots in bridge sphere is one of nice tools to detect  non-perturbed knots which are in $n$-bridge position.
  We  show that if a knot $K$ satisfies the rectangle condition  for $n$-bridge decomposition $(T_1,T_2;S^2)$ of $K$ then the Hempel distance for the $n$-bridge decomposition is greater than or equal to $2$ in Section~\ref{B2}. This implies that if $K$ satisfies the rectangle condition on the $n$-bridge sphere of a $n$-bridge decomposition of $K$ then $K$ is not perturbed. Especially,  Otal~\cite{7} showed that any $n$-bridge presentation of any $2$-bridge knot is perturbed for $n > 2$.
  If $K$  has a $3$-bridge presentation  but the $3$-bridge sphere of a $3$-bridge decomposition  does not allow any perturbation then $K$ is 3-bridge knot. Therefore,  for a  knot $K$ having $3$-bridge decomposition $(T_1,T_2;S^2)$,  if $d(T_1,T_2)\geq 2$ then $K$ is a 3-bridge knot.   Coward~\cite{2} gave a theoretical method to calculate the bridge number of hyperbolic knots. However, we hardly even  know how to find  the bridge numbers of alternating knots having a $3$-bridge presentation. Especially, we wonder whether or not a knot $K$ having a reduced alternating $3$-bridge presentation is a $3$-bridge knot if the presentation is obtained from a reduced alternating presentation of a $3$-bridge knot $K'$ by adding more crossings without violating the alternating condition. We conjecture that they are all $3$-bridge knots. In order to support our conjecture, we investigate the following families.
  In Section~\ref{B6}, we  construct special families of  alternating $3$-bridge  links $N(EAT^{2k+1})$ and $D(EAT^{2k+2})$.  Then,  we show that they are  $3$-bridge links if $k\geq 1$ by using the ``Hexagon parameterization" and a modified ``train track" diagrams based on the Rectangle condition. Since the number of crossings of a reduced alternating knot diagram is the minimal number of crossings of the knot, we know that each family has an infinitely many elements.

\section{$n$-bridge decompositions and  Hempel distance}\label{B2}

Suppose $L$ is a link in $S^3$ and $S^2$ is a $2$-sphere which divides $S^3$ into two $3$-balls $B_1$ and $B_{2}$.
Assume that $L$ intersects $S^2$ transversely.  Let $\tau_i=L\cap B_i=\alpha_{i1}\cup \alpha_{i2}\cup\cdot\cdot\cdot \cup \alpha_{in}$, where $\alpha_{ij}$ are the components of $L\cap B_i$.
We note that $(S^3,L)$ is decomposed into $T_{1}:=(B_{1},\tau_{1})$ and $T_2:=(B_{2},\tau_{2})$ by $S^2$. The triple $(T_1,T_2;S^2)$ is called an $\emph{n-bridge decomposition}$ of $L$ if each $T_{i}$ is a $\emph{rational}$ $n$-tangle.  $T_i$  is said to be  $\emph{rational}$  if there exists a homeomorphism of pairs ${H}: (B^3,\tau_i)\longrightarrow
(D^2\times I,\{p_1,p_2,\cdot\cdot\cdot ,p_n\}\times I)$,  where $I=[0,1]$. Also, 
we say that $L$ is in 
$\emph{n-bridge position}$ with respect to $S^2$ if $L$ has a $n$-bridge decomposition $(T_1,T_2;S^2)$. Then consider the projection of $L$ onto the $xy$-plane so that the projection of $S^2$ is a horizontal line and the projection of $L$ has $n$ maxima and $n$ minima. Then we say that the diagram  is $n$-bridge presentation of $L$.
 \\

Let $\Sigma_{0,2n}=S^2-L$. Then we say that a simple closed curve on $\Sigma_{0,2n}$ is  $\emph{essential}$ if  neither it bounds a disk nor it is boundary parallel to a puncture.
 Also, we say that $D$ is an $\emph{essential disk}$ of $T_{i}$ if $D\subset B_i-\tau_i$ and $\partial D$ is essential in  $\Sigma_{0,2n}$. 
 The essential simple closed curves on $\Sigma_{0,2n}$ form a $1$-complex $\mathcal{C}(\Sigma_{0,2n})$ which is called the $\emph{curve complex}$ of $\Sigma_{0,2n}$. If $n>2$, the vertices of $\mathcal{C}(\Sigma_{0,2n})$ are the isotopy classes of essential simple closed curves on $\Sigma_{0,2n}$ and a pair of vertices spans an edge of $\mathcal{C}(\Sigma_{0,2n})$ if the corresponding isotopy classes can be realized as disjoint curves. 
 We define that $d([\partial D_1],[\partial D_2])$ is the minimal distance between $[\partial D_{1}]$ and $[\partial D_{2}]$ measured in $\mathcal{C}(\Sigma_{0,2n})$ with the path metric. 
Bachman and Schleimer~\cite{0} defined the $\emph{Hempel distance}$ (or just the $\emph{distance}$) of $(T_{1},T_{2};S^2)$ is defined by 

\begin{center}
$d(T_{1},T_{2}):=\min\{d([\partial D_{1}],[\partial D_{2}])| ~D_{i}$ is an essential disk of $T_{i}$ for $i=1,2\}$.
\end{center}

We note that the distance $d(T_{1},T_{2})$ is a finite non-negative integer since the curve complex is connected.\\

 Now, consider a disk $F^k_{i}$ in the ball $B_i$ for $1\leq k\leq n$ so that $(F^k_{i})^\circ\subset B_i^\circ-\tau_i$ and $\partial F^k_{i}=\alpha_{ik}\cup \beta_k$, where $\beta_k$ is a simple arc between the two endpoints of $\alpha_{ik}$ in $\partial B_i$. The disk $F^k_{i}$ is called a \textit{bridge disk}. Then let $\{F^1_{i}, F^2_{i},...,F^n_{i}\}$ be \textit{ a collection of bridge disks} for $T_i$ if $F^k_{i}$ are pairwise disjoint. We note that there exist such disks  since $T_i$ are  rational $3$-tangles.  Let $\{D_{11},D_{12},...,D_{1~2n-3}\}$ be a maximal collection  of pairwise disjoint, non-isotopic essential disks in $B^3-\tau_1$. This is called a $\emph{collection of cut disks}$. Similarly,  we have a collection of cut disks $\{D_{21},D_{22},...,D_{2~2n-3}\}$  for $T_2$.
 We note that there exists a collection of cut disks  $\{D_{i1},D_{i2}...,$ $D_{i~2n-3}\}$ for $T_i$ so that $ (\cup_{j=1}^{2n-3}D_{ij})\cap (\cup_{j=1}^n F^j_{i})=\emptyset$ for $i=1,2$.
 Suppose $L$ is in $n$-bridge position with respect to $S^2$ for $n\geq 3$. If  there exist essential disks $D_1$ and $D_2$ of $T_1$ and $T_2$ respectively such that $D_i$ are cut disks and $[\partial D_1]=[\partial D_2]$, then $L$ is separated by the sphere into an $m$-bridge sublink and an $(n-m)$-bridge  sublink of $L$. We note that $0<m<n$ since $D_i$ is an essential disk in $B^3-\tau_i$. So, if $L$ is a knot then there are no such disks $D_1$ and $D_2$.  
Let $K$ be a knot which is in $n$-bridge position with respect to a sphere $P$.
Suppose there is a pair of bridge disks $E_{1} \subset B_1$ and $E_{2} \subset B_2$ so that the arcs $E_{1} \cap P$ and
$E_{2} \cap P$ intersect precisely at one end. Then $K$ is said to be $\emph{perturbed}$ with respect to $P$ (and
vice verse), and $E_{1},~ E_{2}$ are called $\emph{cancelling disks}$ for $K.$
We note that if there are cancelling disks for $K$ then we can construct collections of cut disks so that $d(T_1,T_2)=1$ as in the proof of Lemma~\ref{T11}.

\begin{Lem}[Ozawa, Takao~\cite{8}]\label{T11}
Suppose that  a knot $K$ ($n>2$) in $n$-bridge position has Hempel distance greater than equal to $2$. Then $K$ is not perturbed with respect to the bridge sphere $S^2$. 
\end{Lem}

\section{Rectangle condition on the $n$-bridge sphere for a $n$-bridge decomposition of a knot}\label{B3}

 A $\emph{n-tangle}$ is the disjoint union of $n$ properly embedded arcs in the unit 3-ball; the embedding must send the endpoints of the arcs to $2n$ marked (fixed) points on the ball's boundary. Without loss of generality, consider the marked points on the 3-ball boundary to lie on a great circle $C$ (or a horizontal line $L$). The tangle can be arranged to be in general position with respect to the projection onto the flat disk or the upper plane in the $xy$-plane bounded by $C$ (or $L$). The projection then gives us a $\emph{tangle diagram}$ $TD$, where we make note of over and undercrossings as with knot diagrams.
Recall that a $n$-tangle $\alpha_1\cup\alpha_2\cup\cdot\cdot\cdot\cup\alpha_n$ in a 3-ball $B^3$, denoted by $T:=(B^3,\alpha=\alpha_1\cup\alpha_2\cup\cdot\cdot\cdot\cup\alpha_n)$, is \textit{rational}   if
 there exists a homeomorphism of pairs
 ${H}: (B^3,\alpha_1\cup\alpha_2 \cup\cdot\cdot\cdot\cup\alpha_n)\longrightarrow
 (D^2\times I,\{p_1,p_2,... ,p_n\}\times I)$, where $I=[0,1]$. \\

Now, let  $K$ be a knot which has a $n$-bridge  decomposition $(T_1,T_2,P)$.
Then, let $\{D_{11},D_{12},...,$ $D_{1~2n-3}\}$ and $\{D_{21},D_{22},...,D_{2~2n-3}\}$ be maximal collections  of  cut disks   for $T_1$ and $T_2$ respectively.
Recall that there exists a collection of bridge disks  $\{F^1_{i},...,F^n_{i}\}$ for $T_i$ so that $(\cup_{j=1}^n F^j_{i})\cap (\cup_{j=1}^{2n-3}D_{ij})=\emptyset$ for $i=1,2$.
For instance, $\{E_1,E_2,E_3\}$ is a maximal collection of essential cut disks for the trivial rational 3-tangle $\epsilon$ as in Figure~\ref{A1}, where trivial means that the projection of the tangle on $xy$-plane has no crossing.
Consider a cut disk $D$ in $T_i$ which may intersect with $\cup_{p=1}^{2n-3}D_{ip}$.
We need to assume that $D$ intersects $\cup_{p=1}^{2n-3}D_{ip}$ transversely and minimally. A subarc $\alpha(D)$ of $\partial D$ cut by $\cup_{p=1}^{2n-3}D_{ip}$ is a \emph{wave}  for the cut disk $D$ in $T_i$   if
there exists an outermost arc $\beta$ in $D\cap(\cup_{p=1}^{2n-3}D_{ip})$ and a corresponding outermost disk $\Delta$ of $D$ with $\partial \Delta=\alpha(D)\cup \beta$.

\begin{figure}[htb]
	\includegraphics[scale=.4]{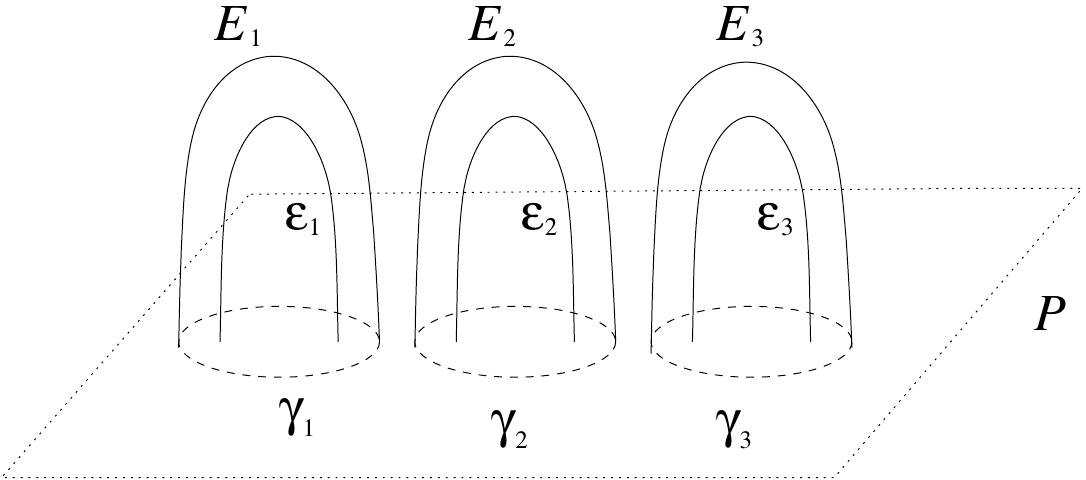}
	\caption{}
	\label{A1}
\end{figure}

\begin{Lem}\label{T2}
Suppose that $\{D_1,...,D_{2n-3}\}$ is a maximal collection of cut disks for a rational tangle $\tau$ and $D$ is an essential disk in $T$. Then if $D$ is not isotopic to any $D_i$ for $i=1,2,...,2n-3$ then $\partial D$ contains a wave. 

\end{Lem}

\begin{proof}
 If $D$ is not isotopic to any $D_j$ and $D\cap(\cup_{i=1}^{2n-3}D_i)=\emptyset$ then $D$ is not essential since $\{D_1,...,D_{2n-3}\}$ is a maximal collection of cut disks. Therefore, $D\cap(\cup_{i=1}^{2n-3}D_i)\neq\emptyset$. Then by taking an outermost arc of them we have a wave $\alpha(D)$ since $D$ intersects $\cup_{i=1}^{2n-3}D_i$ minimally.

\end{proof}

Suppose that $\{D_1,...,D_{2n-3}\}$ be a maximal collection of cut disks for a rational tangle $\tau$. A cut disk $D$ is \textit{essential} if $D$ cuts $\Sigma_{0,2n}$ into a $2$-punctured disk and $(2n-2)$-punctured disk. Then there is a maximal collection of essential cut disks $\mathcal{D}=\{D_1,...,D_n\}\subset\{D_1,D_2,...,$ $D_{2n-3}\}$. Then they cut $\Sigma_{0,2n}$ into $n$   2-punctured disks and a planer surface $S_{\mathcal{D}}$ with $n$ boundary components. We may assume that the collection of essential cut disks $\{D_1,...,D_n\}$ are the first $n$ elements of the maximal collection of cut disks $\{D_1,...,D_n,D_{n+1},...,D_{2n-3}\}$. Then $S_{\mathcal{D}}$ has a pants decomposition with $n-2$ pairs of pants. Suppose that $P$ and $Q$ are pairs of pants of the pants decompostions of $S_{\mathcal{D}}$ and $S_{\mathcal{D}'}$ respectively, where $\mathcal{D}'$ is a maximal collection of essential cut disks for a rational $n$-tangle $\tau'$. We assume that $\partial P$ and $\partial Q$ intersect transversely and minimally. Then we say that the pairs of pants $P$ and $Q$ are \textit{tight} if for each $4$-tuple of nine combinations as below there is a rectangle $R$ embedded in $P$ and $Q$ such that the interior of $R$ is disjoint from $\partial P\cup \partial Q$ and the four edges of $\partial R$ are subarcs of the $4$-tuple, where $\{a_1,b_1,c_1\}$ and $\{a_2,b_2,c_2\}$ are the three boundary components of $P$ and $Q$ respectively.\\

$(a_1,b_1,a_2,b_2)$ \hskip 20pt$(a_1,b_1,a_2,c_2)$ \hskip 20pt$(a_1,b_1,b_2,c_2)$ \hskip 20pt $(a_1,c_1,a_2,b_2)$ \hskip 20pt$(a_1,c_1,a_2,c_2)$\\

 $(a_1,c_1,b_2,c_2)$ \hskip 20pt $(b_1,c_1,a_2,b_2)$ \hskip 20pt$(b_1,c_1,a_2,c_2)$ \hskip 20pt$(b_1,c_1,b_2,c_2)$ \\

Now, I would like to define the rectangle condition for bridge spheres with an analogous definition to the rectangle condition for Heegaard surfaces. For two pants decompositions $\cup_{i=1}^{n-2} P_i$ and $\cup_{j=1}^{n-2} Q_j$ of  $S_{\mathcal{D}_1}$ and $S_{\mathcal{D}_2}$ for $T_1$ and $T_2$ respectively.
Then we say that $\cup_{i=1}^{n-2} P_i$ and $\cup_{j=1}^{n-2} Q_j$   satisfy the \emph{rectangle condition} if all the non-essential pairs of pants $P_i$ and $Q_j$ are tight for $1\leq i,j\leq n-2$. Then the following proposition is the main result about the rectangle condition.

 \begin{Prop}\label{T3}
 Suppose $\cup_{i=1}^{n-2} P_i$ and $\cup_{j=1}^{n-2} Q_j$ are two pants decompositions of $S_{\mathcal{D}_1}$ and $S_{\mathcal{D}_2}$ for $T_1$ and $T_2$ respectively. If $\cup_{i=1}^{n-2} P_i$ and $\cup_{j=1}^{n-2} Q_j$ satisfy the rectangle condition, then the Hempel distance $d(T_1,T_2)\geq 2$.
 \end{Prop} 
 
 \begin{proof}
 	Suppose $\cup_{i=1}^{n-2} P_i$ and $\cup_{j=1}^{n-2} Q_j$ are two pants decompositions of $S_{\mathcal{D}_1}$ and $S_{\mathcal{D}_2}$ for $T_1$ and $T_2$ respectively.
  Let $D_{i1},...,D_{i~2n-3}$ be the disjoint, non-isotopic disks in $B_i-\tau_i$ so that $\cup_{j=1}^{2n-3}\partial D_{ij}=\cup_{j=1}^{n-2}\partial P_j$. 
Suppose that $d(T_1,T_2)<2$. Then we have $d(T_1,T_2)=1$ since $K$ is a knot.
So, there exist essential disks $D_{1}$ and $D_{2}$ in $B^3-\tau_1$ and $B^3-\tau_2$ respectively so that $D_{1}\cap D_{2}=\emptyset$ and $[\partial D_1]\neq [\partial D_2]$.  We assume that  $D_i$ meets $\cup_{j=1}^{n-2}\partial P_j$  transversely and minimally if they intersect.\\

First, assume that $D_{1}\cap (\cup_{j=1}^{2n-3}D_{1j})=\emptyset$. Then we note that $\partial D_1$ is isotopic to a boundary of $\partial P_i$. Otherwise, we should have more than $n-2$ pairs of pants for $T_1$. Then, we isotope $\partial D_1$ so that $\partial D_1\subset P_k$ for some $k\in\{1,2,...,n-2\}$ without changing the condition of disjointness with $D_2$. Let $a_k$ be the boundary component of $P_k$ which is isotopic to $\partial D_1$ without loss of generality. Then we note that there is no path to connect $a_k$ and $b_k$, and $a_k$ and $c_k$ in $S_{\mathcal{D}}$ without meeting $\partial D_1$, where $b_k$ and $c_k$ are the other boundary components of $P_k$. Then there are two subcases for $\partial D_2$ as follows. \\

If $\partial D_2$ is isotopic to one of $\partial Q_j$ then $\partial D_2$ should meet $\partial D_1$ to satisfy the rectangle condition. This contradicts the condition that $D_1\cap D_2=\emptyset$.\\

If $\partial D_{2}$ is  isotopic to none of $\partial Q_j$  then $\partial D_{2}$ has a wave $\alpha(D_2)$ by Lemma~\ref{T2}. We note that the interior of  $\alpha(D_{2})$ does not intersect with $\cup_{j=1}^{n-2}\partial Q_j$ and the two endpoints of $\alpha(D_{2})$ are in the same component of $\partial Q_m$ for some $m\in\{1,2,...,n-2\}$. Let $a_m'$ be the boundary component of $Q_m$ which meets $\alpha(D_2)$.
We note that $\alpha(D_{2})\cup a_m'$ separates  $(\cup_{j=1}^{n-2}\partial Q_j)-a_m'$ into two non-empty  sets of boundary components since $D_{2}$ is an essential disk in $B_2-\tau_2$. Especially, there is a pair of pants $Q_l$ so that the boundary components of $Q_l$ are separated by $\alpha(D_{2})\cup a_m'$ since $\Sigma_{0,2n}$ is connected. Let $a_l'$ and $b_l'$ be two components of $Q_l$ which are separated by $\alpha(D_{2})\cup a_m'$.
So, $\alpha(D_2)$ needs to meet the all $\partial P_i$ for $i=1,2,...,n-2$ and especially there are two points $x,y$ in the interior of $\alpha(D_2)$ so that $x\in a_k$, $y\in b_k$ and  the arc $\beta$ between $x$ and $y$ in $\alpha(D_2)$ only meets  $\cup_{i=1}^{n-2}\partial P_i$ at $x$ and $y$ by the rectangle condition. Since $\partial D_1\cap \alpha(D_2)=\emptyset$, we can take a path $\beta'$ which is closely parallel to $\beta$ so that $\beta'$ connects $a_k$ and $b_k$  in $\Sigma_{0,2n}$ without meeting $\partial D_{1}$. This makes a contradiction.\\

Now, assume that $D_{1}\cap (\cup_{j=1}^{2n-3}D_{1j})\neq \emptyset$.   Take a component $C$ of $D_1-(\cup_{j=1}^{2n-3}D_{1j})$ whose closure is a bigon. Then assume that $C\cap D_{1s}\neq \emptyset$ for $s\in\{1,2,...,2n-3\}$. Then $C\cup D_{1s}$ separates $\{D_{11},...,D_{1s-1},D_{1s+1},...,D_{1~2n-3}\}$ into two non-empty sets since $D_1$ is an essential disk in $B_1-\tau_1$. %Otherwise, either we can  isotope $C$ across $D_{1s}$ or $\partial D_1$ has an infinite spiral.  
So, there exists $k\in\{1,2,...,2n-2\}$ so that the boundary components of $P_k$ are separated by $C\cup D_{1s}$. Let $a_k$ and $b_k$ be the components of $P_k$ so that they are separated by $C\cup D_{1s}$. Then there is no path to connect $a_k$ and $b_k$ in $\Sigma_{0,2n}$ without meeting $\partial D_{1}$ or $\partial D_{1s}$. We note that $\partial D_{1s}$ is isotopic to one of the boundary components $\cup_{i=1}^{n-2}\partial P_i$.\\

 If $\partial D_{2}$ is isotopic to one of $\partial Q_j$ we switch the indices $1$ and $2$ to have a contradiction with the same reason as above.\\
 
 If $\partial D_2$ is isotopic to none of $\partial Q_j$ then $\partial D_2$ has a wave $\alpha(D_2)$.
 Then, by using a similar argument above, there are two points $x,y$ in the interior of $\alpha(D_2)$ so that $x\in a_k$, $y\in b_k$ and  the arc $\beta$ between $x$ and $y$ in $\alpha(D_2)$ only meets $\cup_{i=1}^{n-2}\partial P_i$ at $x$ and $y$ by the rectangle condition. Since $\partial D_1\cap \alpha(D_2)=\emptyset$, we can take a path $\beta'$ which is closely parallel to $\beta$ so that $\beta'$ connects $a_k$ and $b_k$  in $\Sigma_{0,2n}$ without meeting $\partial D_{1}$ or $\partial D_{1s}$. This contradicts the existence of $a_k$ and $b_k$ so that there is no path to connect $a_k$ and $b_k$ in $\Sigma_{0,2n}$ without meeting $\partial D_{1}$ or $\partial D_{1s}$. This completes the proof.
 
 \end{proof}

 Now, we will discuss how to check whether or not given two pants decompositions satisfy the rectangle condition. First of all, we parameterize the boundary of an essential disk in $B^3-\tau_i$ in Section~\ref{B4}.
 
\section{Dehn's parameterization of $\Sigma_{0,6}$} \label{B4}

Let $\gamma$ be a simple closed curve in $\Sigma_{0,6}$. We consider the standard essential disks $E_1,E_2,E_3$ as in Figure~\ref{A1}.
Consider the pair of pants $I:=\Sigma_{0,6}-\{E_1'\cup E_2'\cup E_3'\}$, where $E_i'$ is the two punctured disk in $\Sigma_{0,6}$ so that $\partial E_i'=\partial E_i$ as in  Figure~\ref{A2}. 

\begin{figure}[htb]
	\includegraphics[scale=.4]{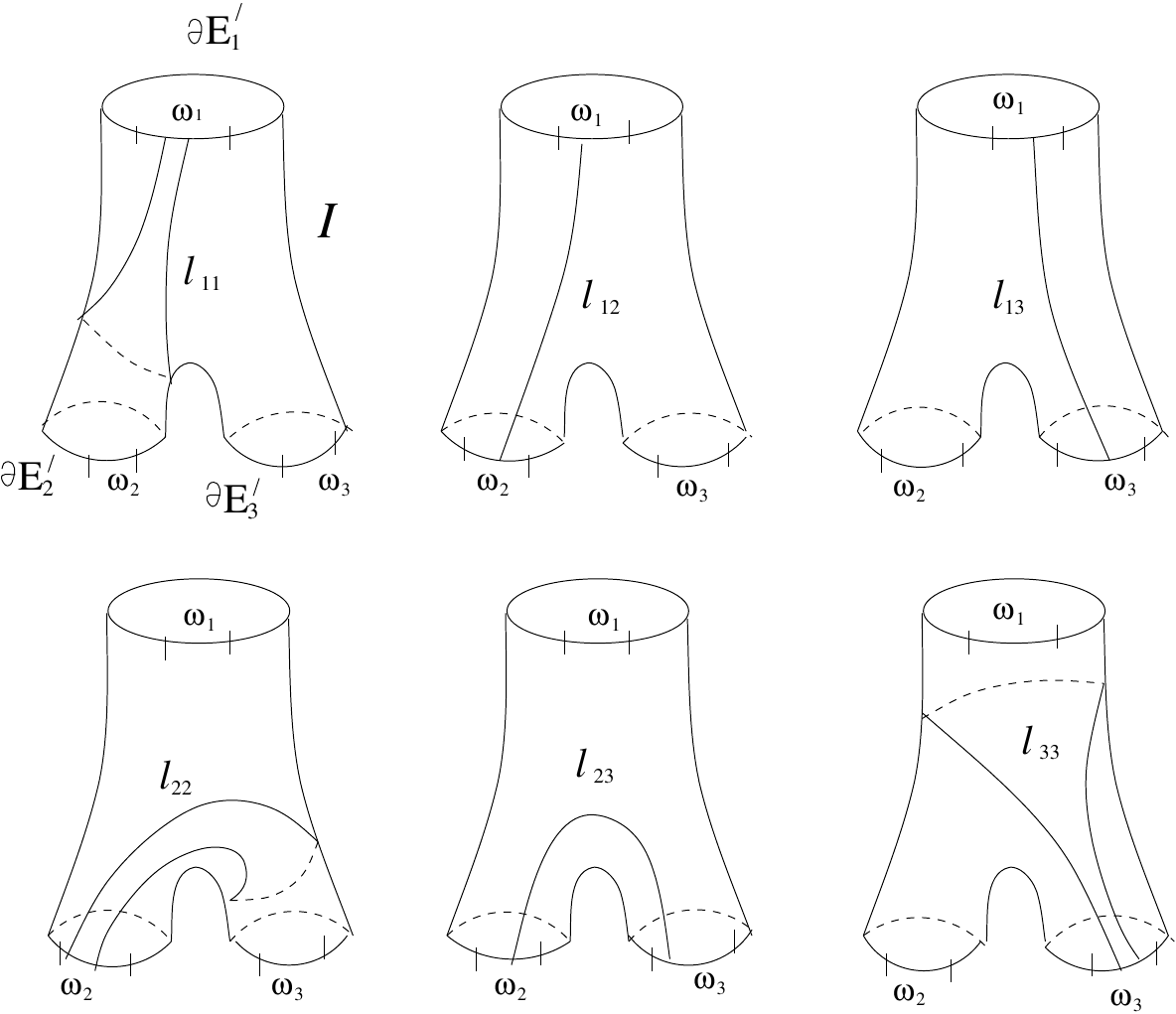}
	\caption{}
	\label{A2}
\end{figure}

Figure~\ref{A2}  shows $\emph{standard arcs}$ $l_{ij}$ in the pair of pants $I$. We notice that we can isotope $\gamma$ into $\delta$ in $\Sigma_{0,6}$ so that each component of $\delta\cap I$ is isotopic to one of the standard arcs and $\delta\cap \partial E_i\subset \omega_i$. 
Then we say that subarc $\alpha$ of $\delta$ is $\emph{carried by}$ $l_{ij}$ if some component of $\alpha\cap I$ is isotopic to $l_{ij}$. The closed arc $\omega_i\subset \partial E_i$ is called a $\emph{window}$. 
 Let $I_i=|\delta\cap \omega_i|$. 
Then $\delta$ can have many parallel arcs which are the same type in $I$. Let $x_{ij}$ be the number of parallel arcs of the type $l_{ij}$ which is called the $weight$ of $l_{ij}.$ 

\begin{Lem}\label{T4}
	
	$I_i$  determine the weights $x_{jk}$ for $i,j,k\in\{1,2,3\}$.
\end{Lem}

\begin{proof} We have two subcases for this.
	First, suppose that $I_i\leq I_j+I_k$ for all distinct $i,j,k\in\{1,2,3\}$. 
	We claim that $x_{11}=x_{22}=x_{33}=0$. If not, then $x_{ii}>0$ for some $i$.
	We notice that $x_{jj}=x_{kk}=0$. So we have $I_i=2x_{ii}+x_{ij}+x_{jk}$, $I_j=x_{ij}$ and $I_k=x_{ik}$.
	This shows that $2x_{ii}+x_{ij}+x_{jk}\leq x_{ij}+x_{ik}$. This makes a contradiction. So, $x_{11}=x_{22}=x_{33}=0$.
	Now, we have $I_i=x_{ij}+x_{ik}$. This implies that
	$x_{ij}={{I_i+I_j-I_k}\over 2}$. \\
	
	Now, suppose that $I_i>I_j+I_k$ for some $i$.
	Then we note that $\gamma$ has $I_i=2x_{ii}+x_{ij}+x_{ik}$, $I_j=x_{ij}$ and $I_k=x_{ik}$.	This implies that $x_{ij}=I_j,~ x_{ik}=I_k$ and $x_{ii}={I_i-I_j-I_k\over 2}$. 
	
\end{proof}

\begin{figure}[htb]
	\includegraphics[scale=.25]{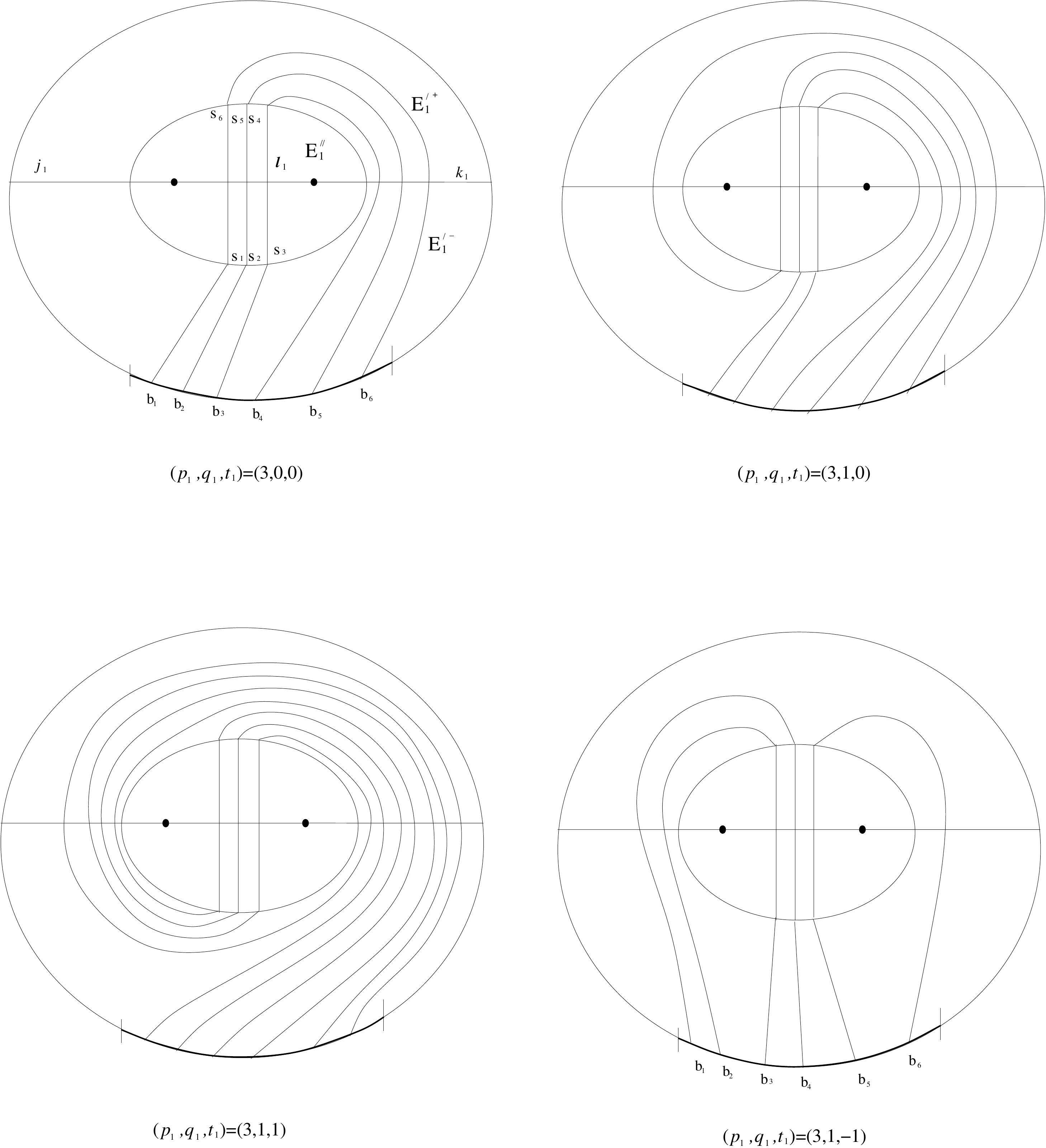}
	\caption{}
	\label{A3}
\end{figure}

Now, we discuss the  arc components of $\delta\cap E_i''$. Let $j_1$ and $k_1$ be the simple arcs as in Figure~\ref{A3}.
We assume that $\partial E_1'\cup\delta \cup j_1\cup k_1\cup l_1$ has no bigon in $\Sigma_{0,6}$. We note that $j_1\cup k_1\cup l_1$ separates $E_1'$ into two semi-disks ${E_1'}^+$ and ${E_1'}^-$ as in Figure~\ref{A3}.
Let ${u}_1^+$ be the number of subarcs of $\delta$ from $l_1$ to $j_1$ in  ${E_1'}^+$. Also, let ${v}_1^+$ be the number of subarcs of $\delta$ from $l_1$ to $k_1$ in ${E_1'}^+$ and let ${w}_1^+$ be the number of subarcs of $\delta$ from $j_1$ to $k_1$ in ${E_1'}^+$.  
Let $m_1=|\delta\cap j_1|$ and $n_1=|\delta\cap k_1|$ in $E_1'$. We note that each component of $\delta\cap E_1'$ meets $l_1$ exactly once. Also, we know that each such component is essential in 
$E_1'-\{1,2\}$. The components of $\delta\cap E_1'$ are determined by three parameters $p_1,q_1,t_1$ as in Figure~\ref{A3},
where $p_1=\min\{|\delta'\cap l_1||\delta'\sim \delta$ in $\Sigma_{0,6}\}$ and $q_1\in \mathbb{Z},~0\leq q_1<p_1$. 
In order to define $q_1$ and $t_1$, consider $m_1$ and $n_1$.
 Then we know that $u_1^++v_1^+=p_1$. 
So, $m_1-n_1=(u_1^++w_1^+)-(v_1^++w_1^+)=u_1^+-v_1^+$. Therefore, $-p_1=-u_1^+-v_1^+\leq u_1^+-v_1^+=m_1-n_1=u_1^+-v_1^+\leq u_1^++v_1^+=p_1$.
So, we know $-p_1\leq m_1-n_1\leq p_1$. Now, we define $q_1$ and $t_1$ as follows. If $n_1-m_1=p_1$ then $q_1\equiv m_1$ (mod $p_1$) and $0\leq q_1<p_1$, and $t_1={m_1-q_1\over p_1}$ and if $-p_1\leq n_1-m_1<p_1$ then $q_1\equiv -m_1$ (mod $p_1$) and $0\leq q_1<p_1$, and $t_1={-m_1-q_1\over p_1}$.
Then $t_1$ is called the $twisting$ $number$ in $E_1'$. 
Let $(p_1,q_1,t_1)$ be the three parameters to determine the arcs in $E_1'$. Similarly, we have the three parameters $(p_i,q_i,t_i)$ for $E_i'$ ($i=2,3$). 
Therefore, $\gamma$ is determined by a sequence of nine parameters $(p_1,q_1,t_1,p_2,q_2,t_2,p_3,q_3,t_3)$ by Lemma~\ref{T4}.\\

Let $\mathcal{C}$ be the set of isotopy classes of simple closed curves in $\Sigma_{0,6}$.
For a given simple closed curve $\delta$ in $\Sigma_{0,6}$, we define $p_i$, $q_i$ and $t_i$ in $E_i'$ as above. Let $q_i'=p_it_i+q_i$ for $i=1,2,3$. Then we have the following Dehn's Theorem.

\begin{Thm} [Special case of Dehn's Theorem ]\label{T5}

There is an one-to-one map $\phi:\mathcal{C}\rightarrow\mathbb{Z}^6$ so that $\phi(\delta)=(p_1,p_2,p_3,q_1',q_2',q_3')$. i.e., it classifies isotopy classes of simple closed curves.
 \end{Thm}

When $p_1=p_2=p_3=0$ then $t_i'=1$ if the simple closed curve is isotopic to $\partial E_i'$ and $t_j'=0$ if $j\neq i$.
Refer~\cite{9} to see the general Dehn's theorem.

 \section{Detecting $3$-bridge knots with $3$-bridge presentations}\label{B5}

 \begin{Thm}\label{T7} Let $(T_1,T_2;S^2)$ be the 3-bridge decomposition of a knot $K$. Suppose $\cup_{i=1}^{n-2} P_i$ and $\cup_{j=1}^{n-2} Q_j$ are two pants decompositions of $S_{\mathcal{D}_1}$ and $S_{\mathcal{D}_2}$ for $T_1$ and $T_2$ respectively.
 If $\cup_{i=1}^{n-2} P_i$ and $\cup_{j=1}^{n-2} Q_j$ satisfy the rectangle condition, then $K$ is a 3-bridge knot.
 \end{Thm}
 
 \begin{proof}
The proof is based on Lemma~\ref{T11} and Proposition~\ref{T3}.
 \end{proof}

To use the rectangle condition to check if a $3$-bridge presentation link $K$ is $3$-bridge link,  we need to choose a level sphere $S^2$ to have the $3$-bridge decomposition $(T_1,T_2;S^2)$ of $K$. Especially we can choose $S'^2$ to have a $3$-bridge decomposition  $(T_\epsilon,T'_2;S'^2)$ of $K$ by the following lemma, where $T_\epsilon=(B^3,\epsilon)$.  (Refer to Figure~\ref{A1}.) However, it does not mean that for a given pants decomposition for $T_{\epsilon}$ there exists a pants decomposition for $T_{2}'$ to satisfy the rectangle condition even if there are two pants decompositions for the two rational tangles $T_{\epsilon}$ and $T_{2}'$ to satisfy the rectangle condition.\\

 \begin{Lem}\label{T8}
 Suppose $(T_1,T_2;P)$ and $(T_\epsilon,T_2';P')$ are $3$-bridge decompositions of a link $K$.
 Then there exist minimal collections of essential cut disks $\mathcal{D}_1$ and $\mathcal{D}_2$ for $T_1$ and $T_2$ respectively so that there exist  two  pants decompositions  $S_{\mathcal{D}_1}$ and $S_{\mathcal{D}_2}$ for $T_1$ and $T_2$ respectively which satisfy the rectangle condition if and only if there exist minimal collections of essential cut disks $\mathcal{D}_1'$ and $\mathcal{D}_2'$ for $T_{\epsilon}$ and $T_2'$ respectively so that there exist two  pants decompositions  $S_{\mathcal{D}'_1}$ and $S_{\mathcal{D}_2'}$ for  $T_\epsilon$ and $T_2'$ respectively which satisfy the rectangle condition.
 \end{Lem}
 \begin{proof}
 	First, we note that the pants decompositions have only one pair of pants. 
 	We consider the self-homeomorphism $H$ of $S^3$ as an extension of a homeomorphims $h$ from $P$ to $P'$ which is obtained from a combination of half Dehn twists to get from $\tau_1$ to $\epsilon$ by following up the bridge presentation. We note that the homeomorphism of $\Sigma_{0,6}$ preserves rectangles to satisfy the rectangle condition. Let $\mathcal{D}_i'=H(\mathcal{D}_i)$. Then it is enough to show that $\partial \mathcal{D}_1'$ and $\partial \mathcal{D}_2'$ bound the minimal collections of cut disks for $T_\epsilon$ and $T_2'$. 
 	By construction of $H$, it is clear that
 	$\partial \mathcal{D}_i'$ bounds essential cut disks. (Refer to \cite{5}.)
 	
 	 \end{proof}

 Unfortunately, there are  infinitely many different maximal collection of essential cut disks in $B^3-\epsilon$. In other words, we have infinitely many different pants decompositions of $\Sigma_{0,6}$ for $T_{\epsilon}$. So, it is impossible to check the all combinations of two pants decompositions. In this paper, we will especially discuss the case that $\mathcal{E}=\{E_1, E_2, E_3\}$ is the collection of maximal essential cut disks for $T_{\epsilon}$. (Refer to Figure~\ref{A1}.) Then, we note that  $\partial E_1\cup \partial E_2\cup \partial E_3$ is the boundaries of $S_{\mathcal{E}}$ which is pants decomposition. In order to get a pants decomposition for $T_2'$, we consider three bridge disks $\mathcal{F}=\{F_1,F_2,F_3\}$ for $T_2'$. Then let $f_i=F_i\cap\Sigma_{0,6}$ which is called \textit{essential} arc. Let $N(f_i)$ be the  regular neighborhoods of $f_i$ so that they are pairwise disjoint. Then by deleting $N(f_i)$ from $\Sigma_{0,6}$ we have a pants decomposition $S_{\mathcal{F}}$ for $T_2'$.
For an easier argument, now we consider three essential arcs $f_i$ which connect two punctures in $\Sigma_{0,6}$ instead of the three simple closed curves for the boundary components of $S_{\mathcal{F}}$.
Then, we want to modify the rectangle condition as follows.\\

Let $(T_1,T_2;P)$ be a $3$-bridge decomposition of a link $L$. 
Let $E_{B_i}^1,E_{B_i}^2,E_{B_i}^3$ be the collection of bridge disks for two rational $3$-tangles $T_1$ and $T_2$ respectively.
Then we take the collection of arcs $e_i^1=E_{B_i}^1\cap P,e_i^2=E_{B_i}^2\cap P,e_i^3=E_{B_i}^3\cap P$.
Then we say that the collections of arcs $\{e_1^1,e_1^2,e_1^3\}$ and $\{e_2^1,e_2^2,e_2^3\}$ satisfy the $\textit{rectangle condition}$ if there is a rectangle $R$ embedded in $P-((\cup_{i=1}^3 e_1^i)\cup(\cup_{j=1}^3 e_2^j))$ such that the interior of $R$ is disjoint from $(\cup_{i=1}^3 e_1^i)\cup(\cup_{j=1}^3 e_2^j)$ and the four edges of $\partial R$ are subarcs of the four entries of each combination as below respectively. \\

$(e_1^1,e_1^2,e_2^1,e_2^2)$ \hskip 20pt$(e_1^1,e_1^3,e_2^1,e_2^2)$ \hskip 20pt$(e_1^2,e_1^3,e_2^1,e_2^2)$ \hskip 20pt $(e_1^1,e_1^2,e_2^1,e_2^3)$\hskip 20pt$(e_1^1,e_1^3,e_2^1,e_2^3)$\\

$(e_1^2,e_1^3,e_2^1,e_2^3)$ \hskip 20pt $(e_1^1,e_1^2,e_2^2,e_2^3)$ \hskip 20pt$(e_1^1,e_1^3,e_2^2,e_2^3)$ \hskip 20pt$(e_1^2,e_1^3,e_2^2,e_2^3)$

\begin{figure}[htb]
	\includegraphics[scale=.6]{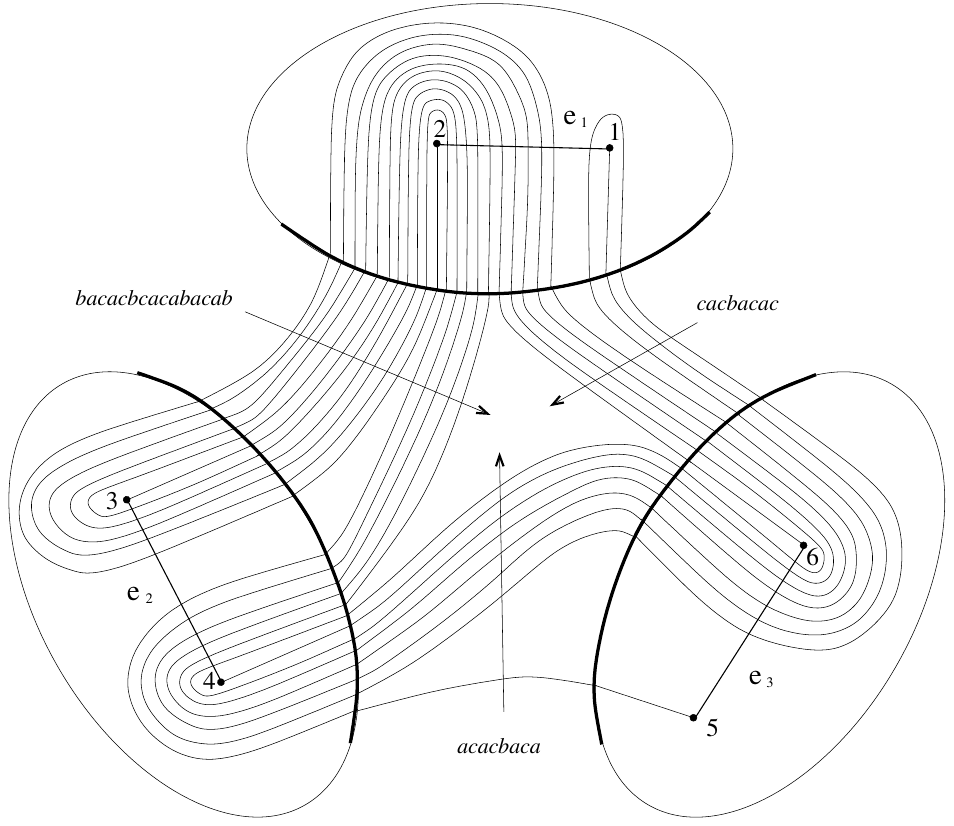}
	\vskip -280pt
	\caption{}
	\label{A4}
\end{figure}

\begin{Prop}\label{T12}  Suppose that $(T_1,T_2;P)$ is a $3$-bridge decomposition of a link $L$. 
	For given  collections of bridge disks $\{E_{B_i}^1,E_{B_i}^2,E_{B_i}^3\}$ for two rational $3$-tangles $T_1$ and $T_2$ respectively,
 take the collection of arcs $e_i^1=E_{B_i}^1\cap P,e_i^2=E_{B_i}^2\cap P,e_i^3=E_{B_i}^3\cap P$. If the collections of arcs $\{e_1^1,e_1^2,e_1^3\}$ and $\{e_2^1,e_2^2,e_2^3\}$ satisfy the rectangle condition in $\Sigma_{0,6}$ then $L=\tau_1\cup\tau_2$ is a $3$-bridge link.
	
\end{Prop}

The proof of Proposition~\ref{T12} is analogous to the proof of Theorem~\ref{T7}.\\

Now, we will discuss how to check whether or not two pants decompositions  satisfy the rectangle condition.
The figure~\ref{A4} shows that  $a\cup b\cup c$ and $e_1\cup e_2\cup e_3$ satisfy the rectangle condition, where  $a$ is the arc from  $2$ to $3$, $b$ is the arc from $1$ to $5$ and $c$ is the arc from $4$ to $6$, and $e_i$ are the straight line which connects two punctures in each two punctured disk $E_i'$.  The  alphabetic sequences are the orders of arcs, where.
 Then we can check that each sequence contains all the possible adjacent pairs $\{a,b\}, \{a,c\}$ and $\{b,c\}$.\\

We recall that $(p,q,t)$ determines the arc pattern in a two punctured disk as Figure~\ref{A5}. In order to analyize the arcs,  take the boundaries of the regular neighborhood of the essential arcs as in Figure~\ref{A5}.  We note that $p$ is the same with the intersection number between the three essential arcs and the window $\omega$. Except the two arcs start from the punctures, each component meets twice with $\omega$. So, we name the arcs by using positive integers $1$ to ${p\over 2}+1$. For example, we name $1,2,3,4,5$ to the arc components as in Figure~\ref{A5}.

\begin{figure}[htb]
\begin{center}
\includegraphics[scale=.35]{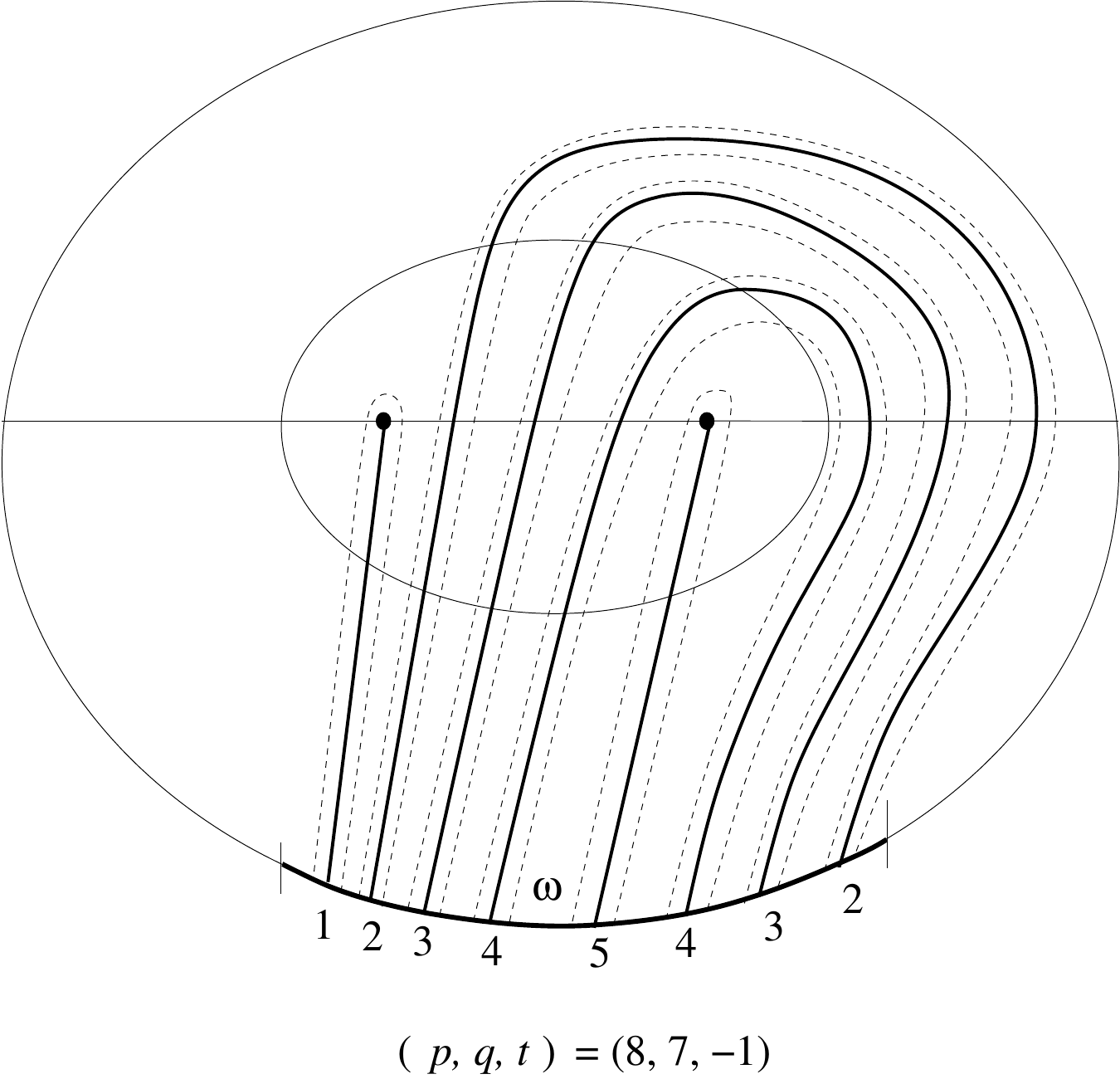}
\end{center}
\caption{}
\label{A5}
 \end{figure}
 
 \begin{Lem}\label{L1}
 $p$ and $q$ determine the pairs of intersection points in $\omega$ which are the two endpoints of each arc component  in 
 a two punctured disk as follows.\\
 
 \begin{enumerate}
 \item If $1\leq j\leq  {p-q-1\over 2}$ then replace $p-q+1-j$ by $j$.\\
 
 \item  If ${2p-q+1\over 2}< j\leq p$ then replace $j$ by $2p-q+1-j$. After this procedure, replace $j$ by $j-m_2$ if $2m_1+2\leq j\leq p$.
 \end{enumerate}
 \end{Lem}
 
 \begin{figure}[htb]
\begin{center}
\includegraphics[scale=.4]{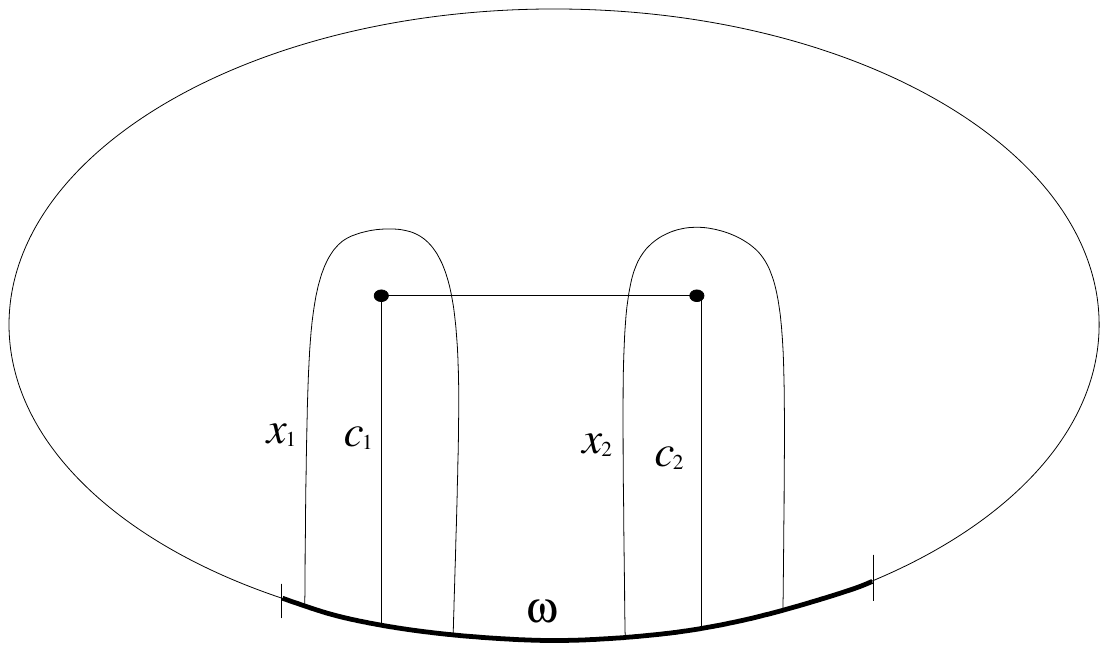}
\end{center}
\vskip -200pt
\caption{}
\label{A6}
 \end{figure}
 
\begin{proof}
First, we fix $t=-1$  since taking a different twisting number $t$ does not change the numbering of the intersection at $\omega$. Then we get a diagram as in Figure~\ref{A6}. Let $c_1$ and $c_2$ be the arc components start from the punctures $1$ and $2$ respectively. Then we note that the other arcs are isotopic to either $x_1$ or $x_2$ since $t=-1$. Let $m_i$ be the number of arcs which are isotopic to $x_i$ for $i=1,2$. Then we note that $q'=q+(-1)p=(-2m_1-1)$ by considering the connecting pattern of arcs between $e$ and $\omega$. Therefore, $m_1=(p-q-1)/2$.
 We also note that $2m_1+2m_2+2=p$. Therefore, we also get $m_2=(q-1)/2$. Then,
 the following is the cases we need to consider.
 
 \begin{enumerate}
 \item If $1\leq j\leq m_1$ then replace $2m_1+2-j$ by $j$ to give the same name to the other endpoint of the arc having the endpoint $j$.
 
 \item If $ 2m_1+m_2+2< j\leq p_i$ then replace $j$ by $2(2m_1+m_2+2)-j$ to give the same name to the other endpoint of the arc having the endpoint $j$.
 \end{enumerate}
 
 Then subtract $m_2$ from $j$ if $2m_1+2\leq j\leq p$ to assign the arcs to the arithmetic sequence $\{1,2,...,m_1+m_2+2\}$.
 This completes the proof.
 
\end{proof}
 
 We note that there are ${p_i+2\over 2}$ arc components in $E_i'$.
 In order to give a better numbering to the arcs in $E_1'\cup E_2'\cup E_3'$, we need to use the next positive integers to name the arcs in $E_j'$ for $j\neq i$. Then we just use positive integers from $1$ to ${p_1+p_2+p_3+6\over 2} $.
We note that we can eliminate the case that $x_{ii}>0$ for some $i$ since  if $x_{ii}>0$ for some $i$ then $\{a,b,c\}$ and $\{e_1, e_2,e_3 \}$  do not satisfy the rectangle condition.
 
 \begin{Prop}
$\{p_1,q_1,p_2,q_2,p_3,q_3\}$ determines the satisfaction of the rectangle condition.
 \end{Prop}
 
 \begin{proof} 
 	By Lemma~\ref{L1}, we can assign positive integers to the arcs in $E_i$ by 
$\{p_1,q_1,p_2,q_2,p_3,q_3\}$. Also, The weights $x_{ij}$ for the standard arcs are determined by $I_i=p_i$ by referring to Lemma~\ref{T4}.	We may assume that $x_{ij}>0$ if $i\neq j$ and  $x_{ii}=0$ for all $i$. If not,  then we easily can check that  $\{a,b,c\}$ and $\{e_1, e_2,e_3 \}$  do not satisfy the rectangle condition.
Let $(\alpha_1,\alpha_2,...,\alpha_{p_1}), (\beta_1,\beta_2,...,\beta_{p_2})$ and $(\gamma_1,\gamma_2,...,\gamma_{p_3})$ be the ordered numbering from the left to the right at $\omega_i$.
 Then we have equalities to connect the intersection points at $\omega$ in $I$ by $x_{ij}$. i.e., $\alpha_i=\beta_{p_2+1-i}$ for $1\leq i\leq x_{12}$, $\gamma_j=\alpha_{p_1+1-j}$ for $1\leq j\leq x_{13}$ and $\beta_k=\gamma_{p_3+1-k}$ for $1\leq k\leq x_{23}$. Then we define the equivalent classes as follows. Two positive integers are \textit{equivalent} if there exist finite equalities to match the two numbers. Then we note that there are exactly three equivalence classes. If there are more than three classes, then there exists an equivalence class which makes
 a loop in $\Sigma_{0,6}$ since $p_1+p_2+p_3$ is finite. This contradicts the assumption that there are exactly $3$ arcs. If there are less than three classes then this also contradicts the assumption that there are exactly three essential arcs. Then by assigning three representative numbers to the arcs in $I$ we can check whether or not they satisfy the rectangle condition.
 \end{proof}
 
Now, we investigate the diagram in Figure~\ref{A4} for an example.
  We note that $(p_1,q_1,p_2,q_2,p_3,q_3)$ $=(24,3,24,11,16,1)$ and $t_i=-1$ for all $i$.
Since $m_1={24-3-1\over 2}=10$ and $m_2={3-1\over 2}=1$ for $E_1'$, we have $1=21$, $2=20$, $3=19$, ..., $11$ and $22=24=12$ by $22-10=12$ and $23=13$ by $23-10=13$. Therefore, we have a numbering for $E_1'$ as $(1,2,3,4,5,6,7,8,9,10,11,10,9,8,7,6,5,$ $4,3,2,1,12,13,12)$. Similarly, we have numberings for $E_2'$ and $E_3'$ as $(14,15,16,17,18,19,20,19,$ $ 18,17,16,15,14,21,22,23,24,25,26,25,24,23,22,21)$ and $(27,28,29,30,31,32,33,34,33,32,$ $31,30,29,28,27,35)$. Also, we have $x_{12}=16, x_{13}=8$ and $x_{23}=8$ from $(p_1,p_2,p_3)=(24,24,16)$. Then we have\\
 
 $(1,2,3,4,5,6,7,8,9,10,11,10,9,8,7,6)=(21,22,23,24,25,26,25,24,23,22,21,14, 15,16,$ $17,18)$ and $(27,28,29,30,31,32,33,34)=(12,13,12,1,2,3,4,5)$, and
  $(14,15,16,17,18,19,20,$ $19)=(35,27,28,29,30,31,32,33)$.\\

Therefore, we have three equivalent classes, $a:=13=28=16=8=24=4=33=19=31=2=22=10=14=35$, $b:=11=21=1=30=18=6=26$ and $c:=20=32=3=23=9=15=27=12=29=17=7=25=5=34$.\\

By assigning the prepresentatives $a,b,c$ to the equivalent positive integers,
 we have
 
  $(x_{12},bacacbcacabacacb)$, $(x_{13},cacbacac)$ and $(x_{23},acacbaca)$.
This implies that the example satisfies the rectangle condition.

\section{A family of $3$-bridge links}\label{B6}

Let $Q$ be the flat disk in the $xy$-plane bounded by the great circle $C$  so that a rational 3-tangle $T$ is arranged to be in general position with respect to the projection onto $Q$ as in Figure~\ref{A7}. Emert and Ernst~\cite{3} defined an interesting family of alternating $3$-tangles which are called \textit{essential}. (Refer to~\cite{3}.) Let $C_0$ be the nested circle in $Q$ so that it encloses the trivial $3$-tangle as in Figure~\ref{A7}. Then let $C_1,...,C_n=\partial Q$ be a sequence of nested circles in $Q$ so that each $C_i$ contains a rational $3$-tangle and the annulus $A_i$ bounded by $C_{i-1}$ and $C_i$ contains at most three different twisting patterns as in Figure~\ref{A7}. The three twisting patterns in $A_i$ are as follows. Let $P_j^i$ be the intersection points between $T$ and  $C_i$ in $Q$ for $1\leq i\leq n$ and $1\leq j\leq 6$ as in Figure~\ref{A7}.
 Let $t_i^j$ be the twisting number between the two substrings of $T$ whose endpoints are $\{P_{i}^{j-1},P_{i+1}^{j-1},P_{i}^j,P_{i+1}^{j}\}$. We have $t_j^i=0$ if $i+j$ is an odd number. So, there possibly exist at most three twisting patterns in $A_i$. We say that an alternating rational $3$-tangle is $\emph{essential}$ if  each twisting numbers $t_i^j$ are nonzero when $i+j$ is an even number. (Refer to Figure~\ref{A7}.) Let $EAT^n$ be the set of essential alternating rational $3$-tangles.
  Especially, we say that an essential alternating  rational $3$-tangle is $\emph{special}$, denoted by  $SAT^n$, if  $t_i^j=1$ when $i+j$ is an even number. (Refer to Figure~\ref{A8}.)
   Then we consider two types of closure for a given tangle $T$ to have a link $L$ as follows.
  \begin{figure}[htb]
  	\begin{center}
  		\includegraphics[scale=.4]{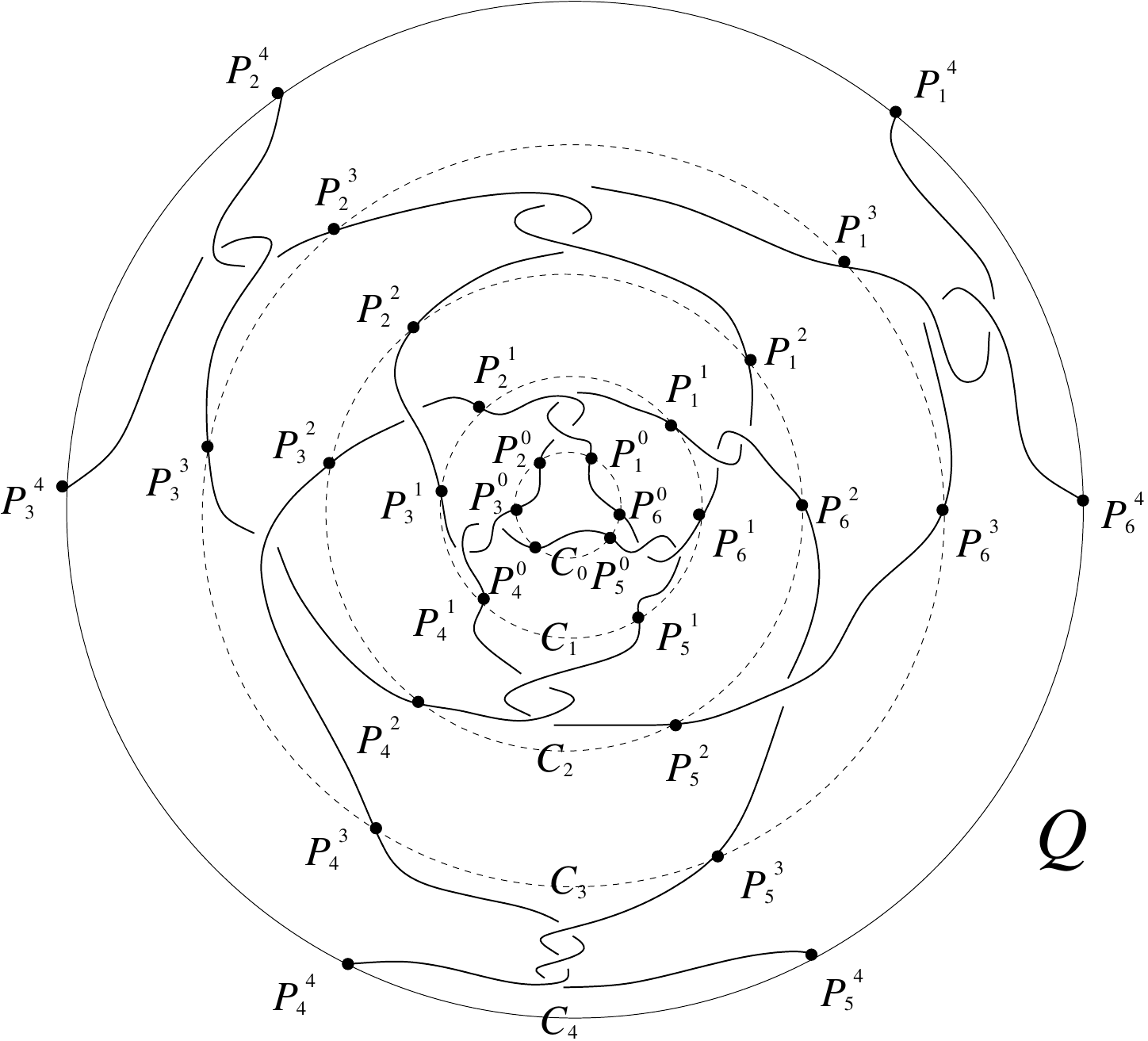}
  	\end{center}
  	\caption{A positive essential rational 3-tangle}
  	\label{A7}
  \end{figure}
 \begin{enumerate}
 	\item Numerator  closures of $T$ denoted by $N(T)$: connect the two points of pairs $(P_1^n,P_2^n),$ $(P_3^n,P_4^n),(P_5^n,P_6^n)$ in $\overline{Q^c}$ with unknotted arcs as in the diagrams $(b)$ and $(d)$ of Figure~\ref{A8}.
 	\item Denominator closures of $T$ denoted by $D(T)$: connect the two points of pairs $(P_2^n,P_3^n),$ $(P_4^n,P_5^n),(P_6^n,P_1^n)$ in $\overline{Q^c}$  with unknotted arcs as in the diagrams $(a), (c)$ and $(e)$ of Figure~\ref{A8}.
 \end{enumerate}
  \begin{figure}[htb]
  	\begin{center}
  		\includegraphics[scale=.2]{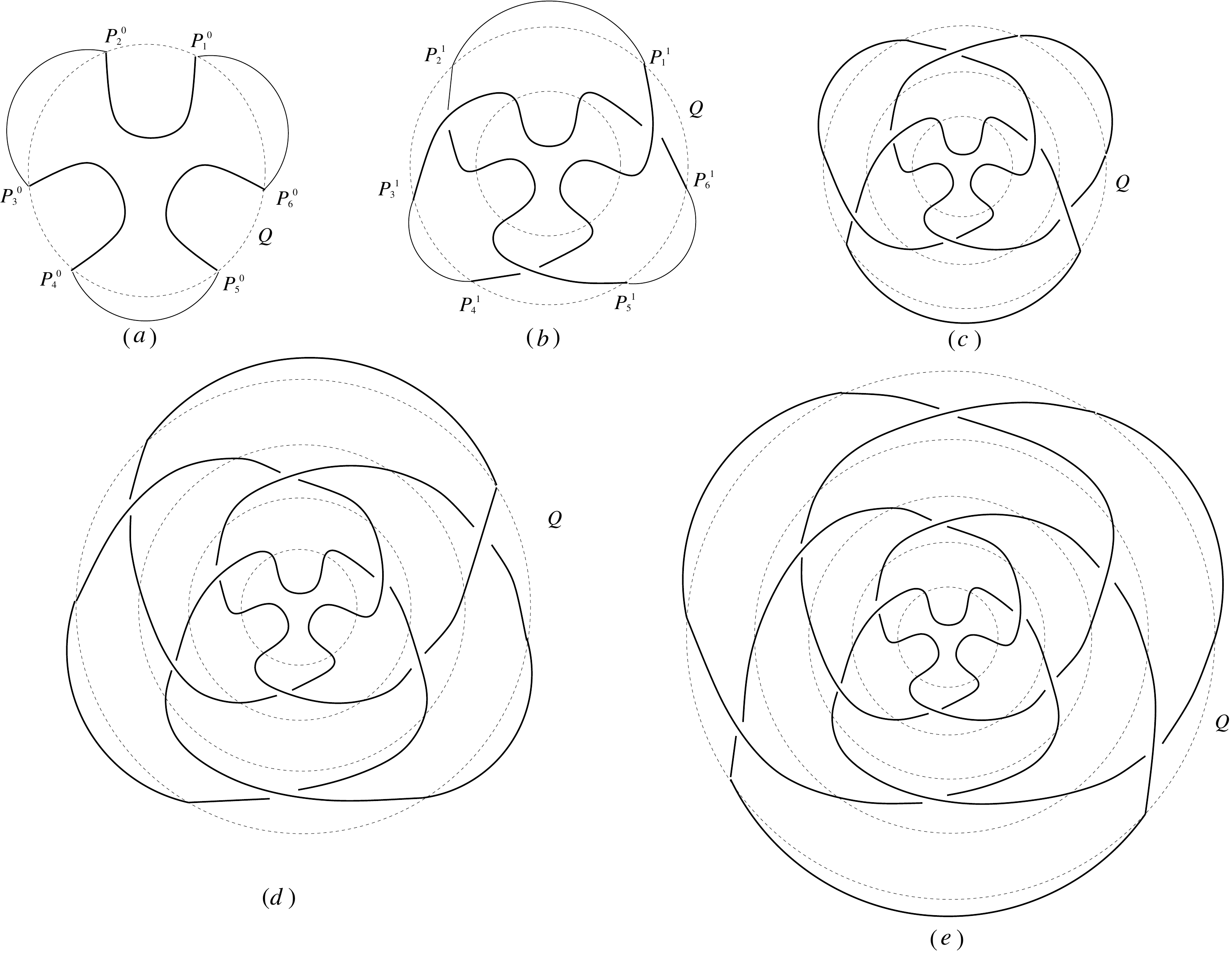}
  	\end{center}
  	\caption{Special essential alternating rational $3$-tangles and their closures}
  	\label{A8}
  \end{figure}

\vskip 10pt
 We note that there are five different $\emph{trivial}$ closures of $T$ and the numerator and the denominator closures are two of them, where $\emph{trivial}$ means  that the rational $3$-tangle for a closure of $T$ in ${Q^c}$ does not have any crossings. Then, we note that $N(SAT^{2k+1})$ and $D(SAT^{2k})$ are prime links by Theorem~\ref{T-1}.

\begin{Thm}[Menasco~\cite{11}]\label{T-1}
Suppose $L$ is a link that has an alternating diagram $D$. Then $L$ is a prime link if and only if $D$ is a prime diagram.
\end{Thm}

Also, the following lemma help to know whether or not $N(SAT^{2k+1})$ or $D(SAT^{2k})$ is a knot.
 
 \begin{Lem}\label{T1}
 $N(SAT^{6k-5})$, $N(SAT^{6k-3})$, $D(SAT^{6k-6})$ and $D(SAT^{6k-2})$ are  knots for $k\in\mathbb{Z}_+$. Moreover, $N(SAT^{6k-1})$ and $D(SAT^{6k-4})$ are links with three components for $k\in\mathbb{Z}_+$.
 \end{Lem}
 
 \begin{proof}
 From the ordered intersection points $(P_{1}^{0},P_{2}^{0},P_{3}^{0},P_{4}^{0},P_{5}^{0},P_{6}^{0})$, we get the ordered intersection points $(P_{6}^{1},P_{3}^{1},P_{2}^{1},P_{5}^{1},P_{4}^{1},P_{1}^{1})$ if we consider the connectivity of the arcs between $C_0$ and $C_1$. Similarly, we have  four more ordered sequences for the connectivity of the arcs through from $C_1$ to $C_5$. These are $(P_{5}^{2},P_{4}^{2},P_{1}^{2},P_{6}^{2},P_{3}^{2},P_{2}^{2})$, $(P_{4}^{3},P_{5}^{3},P_{6}^{3},P_{1}^{3},P_{2}^{3},P_{3}^{3})$, $(P_{3}^{4},P_{6}^{4},P_{5}^{4},P_{2}^{4},P_{1}^{4},P_{4}^{4})$ and $(P_{2}^{5},P_{1}^{5},P_{4}^{5},P_{3}^{5},P_{6}^{5},P_{5}^{5})$.
 Then we note that we have $(P_{1}^{6},P_{2}^{6},P_{3}^{6},$ $P_{4}^{6},P_{5}^{6},P_{6}^{6})$ for the connectivity of arcs between $C_5$ and $C_6$. Moreover, we note that the lower indices of the sequence is the same with the first ordered sequence $(P_{1}^{0},P_{2}^{0},P_{3}^{0},P_{4}^{0},P_{5}^{0},P_{6}^{0})$. Therefore, by considering the six subcases from $n=0$ to $n=5$, we have this lemma.
 \end{proof}

Let $(SAT^{2k+1},T^N_{\epsilon},P)$ and $(SAT^{2k},T^D_{\epsilon},P)$ be the $3$-bridge decompositions of $N(SAT^{2k+1})$ and $D(SAT^{2k})$ respectively, where $T_{\epsilon}^N$ and $T_{\epsilon}^D$ are the trivial tangles in $\overline{Q^c}$ to have numerator closure and denominator closure of a tangle respectively as in Figure~\ref{A8}.
Let $\epsilon^N=\{\epsilon_1^N,\epsilon_2^N,\epsilon_3^N\}$ and $\epsilon^D=\{\epsilon_1^D,\epsilon_2^D,\epsilon_3^D\}$ be the collection of trivial red arcs   for $T^N_{\epsilon}$ and $T^D_{\epsilon}$ respectively as in Figure~\ref{A9}.
Then we note that the diagrams $(d)$ and $(e)$ of Figure~\ref{A9} shows us that $N(SAT^3)$ and $D(SAT^4)$ are $3$-bridge links since the collections of arcs satisfy  the rectangle condition.
\begin{figure}[htb]
	\begin{center}
		\includegraphics[scale=.2]{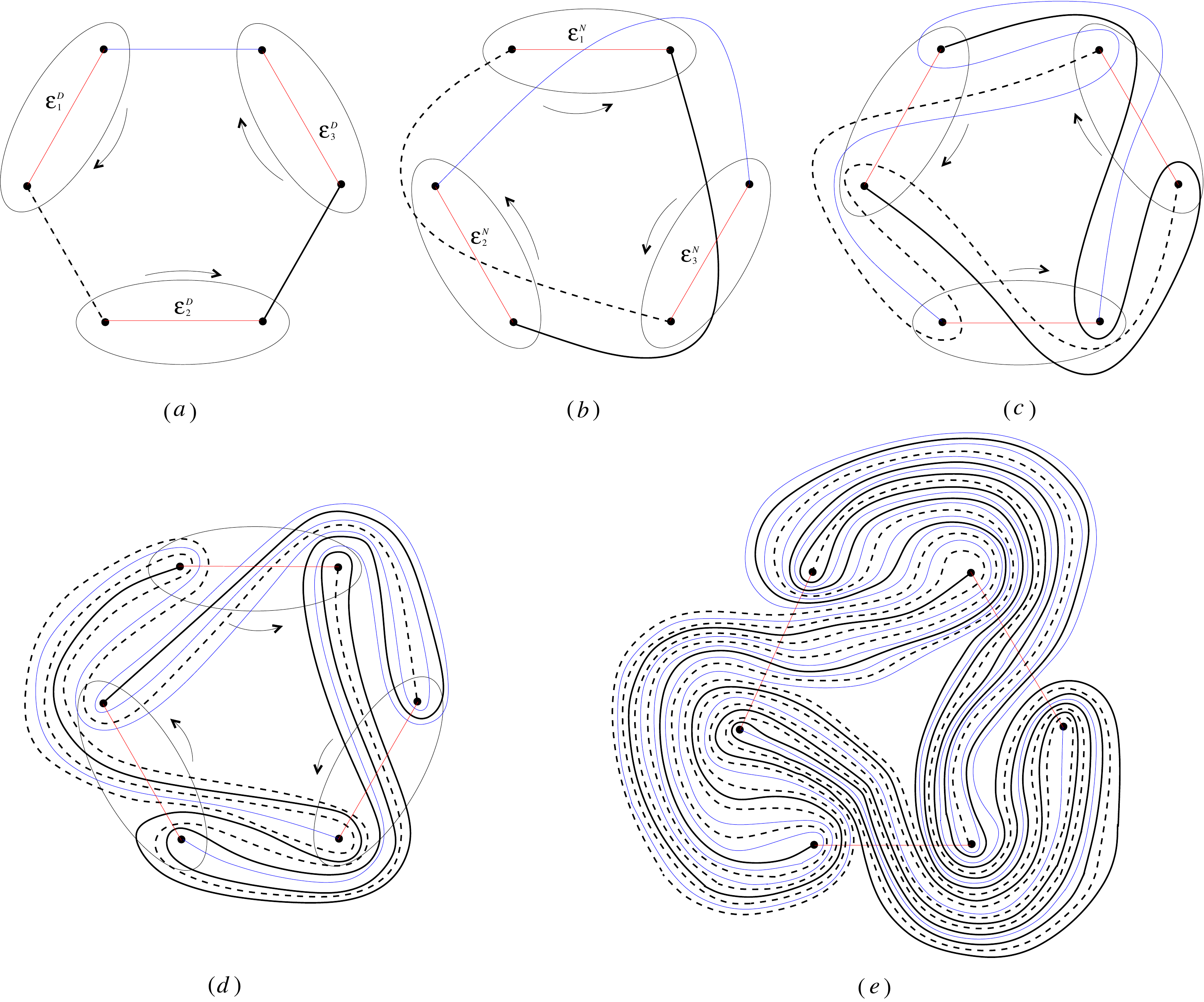}
	\end{center}
	\caption{}
	\label{A9}
\end{figure}
Now, we consider the hexagon $H$ so that $\partial H =\{\epsilon_1^N,\epsilon_2^N,\epsilon_3^N,\epsilon_1^D,\epsilon_2^D,\epsilon_3^D \}$ as Figure~\ref{A10}. It is called the \textit{Hexagon diagram}. Let $a_i$ be the name of the arcs instead of $\epsilon_j^N$ or $\epsilon_k^D$ as in Figure~\ref{A10}. Then, we can define the \textit{weights}  of a simple closed curve $\gamma$ with respect to $H$ as follows. Let $x_{ij}$ be the number of parallel arcs of $\gamma$ from $a_i$ to $a_j$ in $H$. Also, let $x^{kl}$ be the number of parallel arcs of $\gamma$ from $a_k$ to $a_l$ in $H^c$.

\begin{figure}[htb]
	\begin{center}
		\includegraphics[scale=.3]{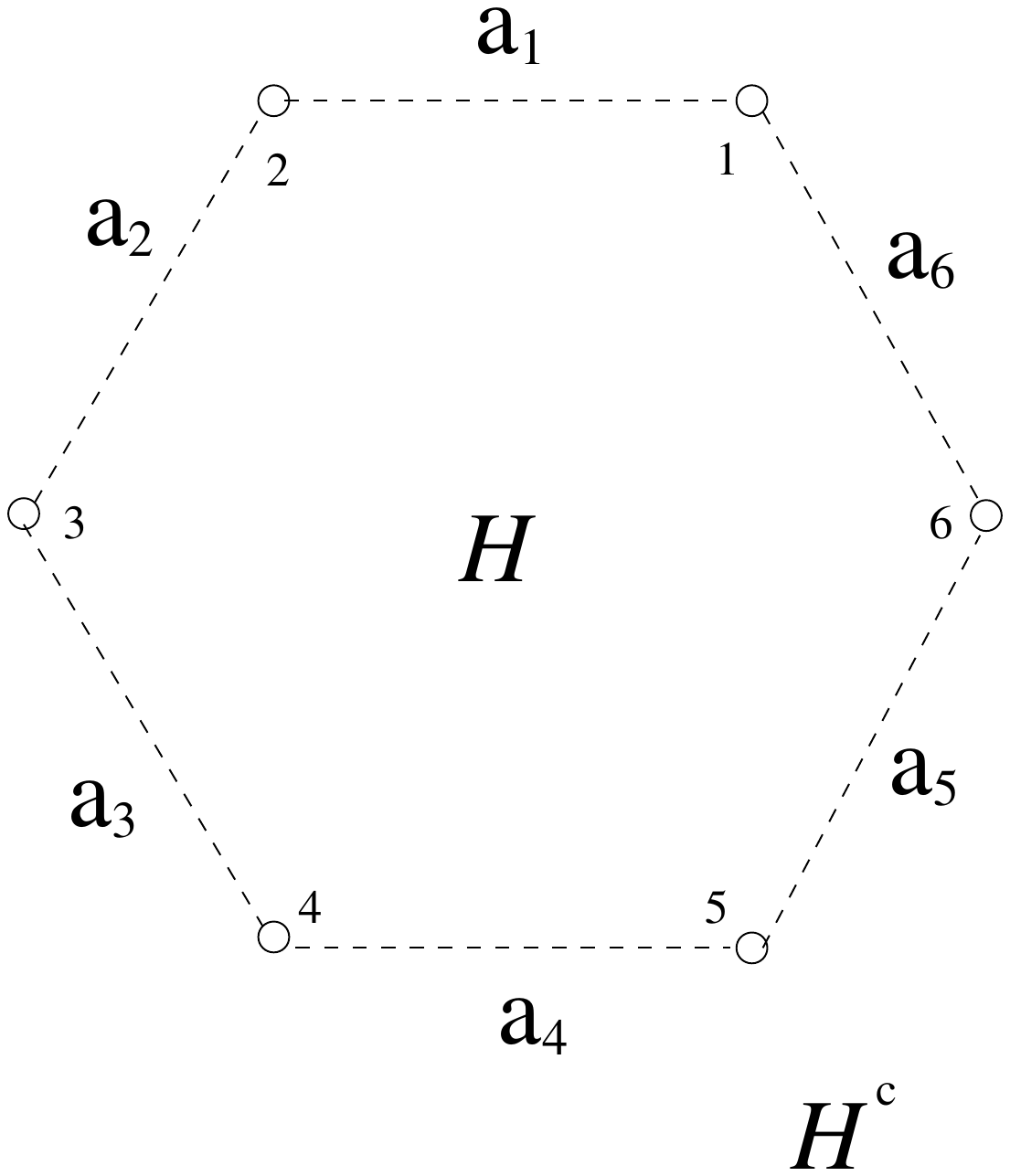}
	\end{center}
	\vskip -50pt
	\caption{}
	\label{A10}
\end{figure}

\begin{Thm}[Kwon~\cite{5}, Lemma 4.3]\label{T-3}
	Suppose that a simple closed curve $\gamma$ is in minimal general position with respect to $\partial H$, where `minimal' means that $\gamma$ and $\partial H$ do not allow to have any bigon. Then the weights $x_{ij}$ and $x^{kl}$ are well defined.
\end{Thm}

We note that $x_{ii}=0$ since $\gamma$ is minimal general position with respect to $\partial H$. Especially, we~\cite{5} made a certain formulas for the weight changes when we apply the half Dehn twists supported on the six  $2$-punctured disks given in the diagrams $(a)$ and $(b)$ of Figure~\ref{A9}. When we apply a half Dehn twist $h$ supported on a $2$-punctured disk, the arcs for $x_{ij}$ or $x^{kl}$ for some $i,j,k,l\in\{1,2,3,4,5,6\}$ in $H$ or $H^c$ possibly become  the arcs for $y_{mm}$ or $y^{nn}$ for some $m, n\in\{1,2,3,4,5,6\}$, where $y_{ij}$ and $y^{kl}$ be the weights of $h(\gamma)$. Then we isotope the arcs to have $y_{mm}=0$ or $y^{nn}=0$ to have new well defined weights. Then, we say that $x_{ij}$ or $x^{kl}$ is \textit{vanished} (by the half Dehn twist $h$). 

\begin{Lem}\label{T-2}
	Suppose that $x_{16}=x_{23}=x_{45}=x^{12}=x^{34}=x^{56}=0$. Then if we apply a half Dehn twist supported on $E_i'$ counterclockwise for some $i$ then none of weights $x_{ij}$ and $x^{kl}$ is vanished. Moreover,  $y_{16}=y_{23}=y_{45}=y^{12}=y^{34}=y^{56}=0$ after applying a sequence of half Dehn twists supported on $E_1', E_2'$ or $E_3'$, where $y_{ij}$ and $y^{kl}$ are the new weights after the sequence of half Dehn twists.
	\end{Lem}
	
	\begin{proof}
		Refer to Theorem 4.4 in \cite{5}.
		\end{proof}

We note that Theorem~\ref{T-3} and Lemma~\ref{T-2} are dealing with a simple closed curve not a simple arc. However, if we take the bounary of a proper regular neighborhood of the simple arc they are useful to prove my following assertions.

\begin{Thm}\label{T0}
	$N(SAT^{2k+1})$ and $D(SAT^{2k+2})$ are $3$-bridge links, where $k$ is a positive integer.
\end{Thm}
\begin{proof}

	By the diagram $(d)$  of Figure~\ref{A9}, we note that $N(SAT^3)$ has two pants decompositions to satisfy the rectangle condition. Therefore,  $N(SAT^3)$ is a $3$-bridge link. Now, take  regular neighborhoods of the three essential arcs (blue, black, dotted black) in $(d)$ of Figure~\ref{A9} so that they are pairwise disjoint and the boundaries of the regular neighborhood do not make any bigon with $\partial H$. Now, we consider a modified train track as in Figure~\ref{A11}. The first diagram shows a modified train track which carries the boundries of the regular neighborhood of the three arcs.   Especially, we note that $x_{16}=x_{23}=x_{45}=x^{12}=x^{34}=x^{56}=0$. We note that none of the weights for the given blue arc types in the first diagram is zero.
	\begin{figure}[htb]
		\begin{center}
		\includegraphics[scale=.5]{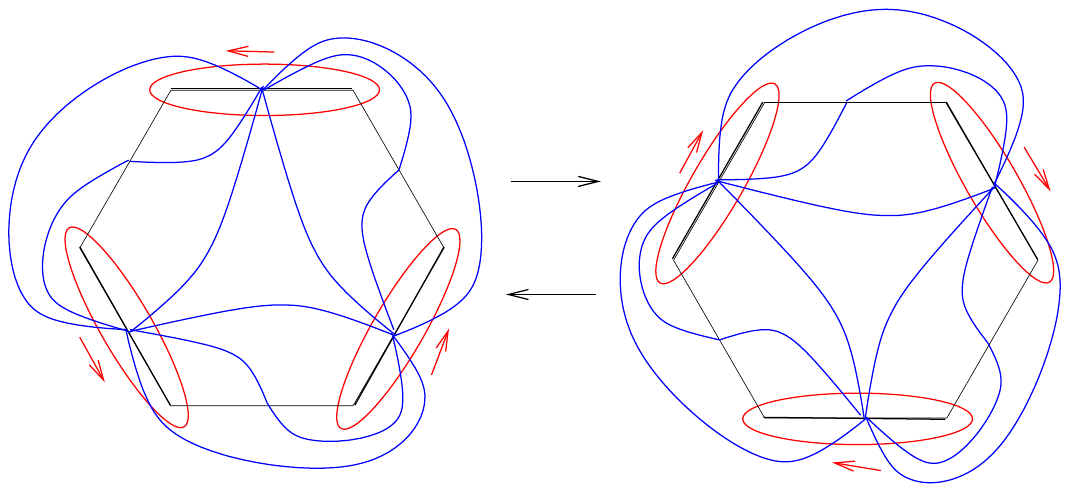}
\end{center}
\vskip -300pt
	\caption{}
	\label{A11}
	\end{figure}
	
	Then the second diagram of Figure~\ref{A11} is the newly obtained train track diagram after applying the three half Dehn twists supported on $E_i'$ counterclockwise respectively. We note that none of the weights for the given blue arc types  in the second diagram is zero. By Lemma~\ref{T-2}, none of weights $x_{ij}$ and $x^{kl}$ of the three essential arcs is vanished. Therefore, in order to show $D(SAT^4)$ is a $3$-bridge link, it is enough to show the arc types representing the rectangles to satisfy the rectangle condition with $\{\epsilon_1^N,\epsilon_2^N,\epsilon_3^N\}$ become  arc types representing the rectangles to satisfy the rectangle condition with $\{\epsilon_1^D,\epsilon_2^D,\epsilon_3^D\}$. Especially, by using the symmetry of the diagram, it is enough to check the three rectangles between $\epsilon_1^N$ and $\epsilon_2^N$.  We note that there are more than one path to carry a rectangle type to satisfy the rectangle condition. However, we point out that all the necessary rectangles to satisfy the rectangle condition are  carried by the arc $a$ or $b$ in Figure~\ref{A12}. Then after applying the three half Dehn twists supported on $E_i'$ counterclockwise, the rectangles are carried by the arc $c$ in Figure~\ref{A12}. Then they are the rectangles to satisfy the rectangle condition between  $\epsilon_1^D$ and $\epsilon_3^D$. The other two cases make the rectangles between $\epsilon_1^D$ and $\epsilon_2^D$, and $\epsilon_2^D$ and $\epsilon_3^D$. This implies that $D(SAT^4)$ is a $3$-bridge link. To show $N(SAT^5)$ is a $3$-bridge link, we note that the first diagram is the newly obtained train track diagram from the second diagram for $D(SAT^4)$ by applying the three half Dehn twists supported on the three two punctured disks in the second diagram in Figure~\ref{A11} clockwise. By using a similar argument, we can show that $N(SAT^5)$ is also a $3$-bridge link. By repeating this argument, we show that $N(SAT^{2k+1})$ and $D(SAT^{2k+2})$ are $3$-bridge links for all $k$.
		\begin{figure}[htb]
			\begin{center}
				\includegraphics[scale=.5]{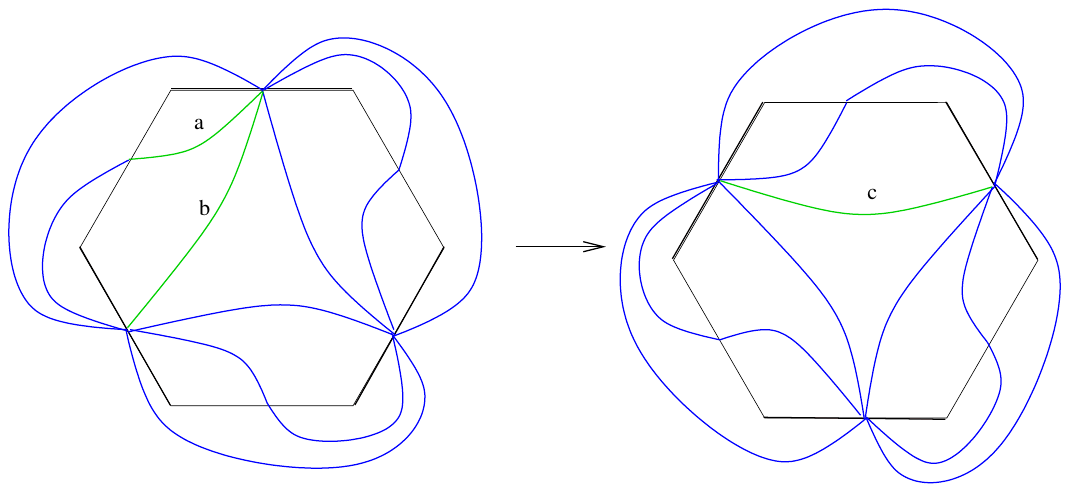}
			\end{center}
			\vskip -300pt
			\caption{}
			\label{A12}
		\end{figure}
	\end{proof} 
	
	\begin{Cor}
	    $N(EAT^{2k+1})$ and $D(EAT^{2k+2})$ are $3$-bridge links if $k\geq 1$.
	\end{Cor}
	\begin{figure}[htb]
		\begin{center}
			\includegraphics[scale=.5]{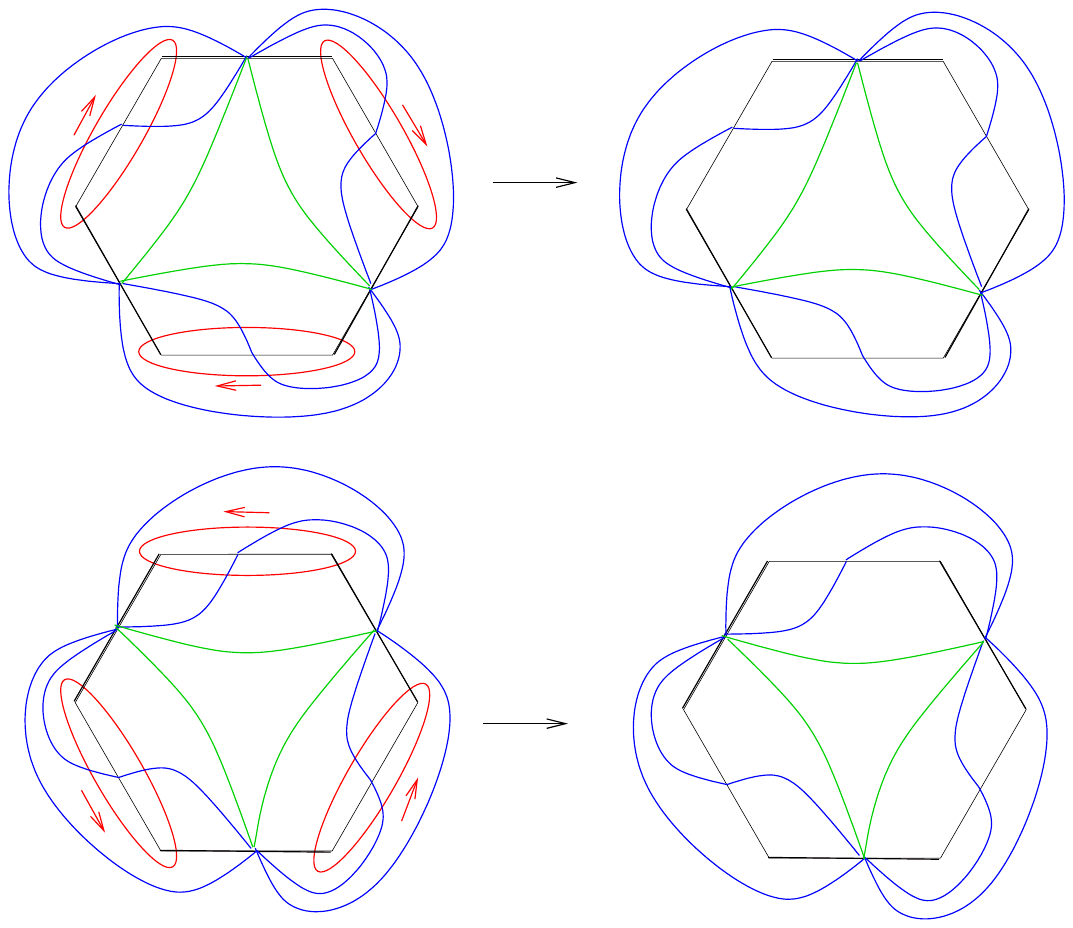}
		\end{center}
		\vskip -200pt
		\caption{}
		\label{A13}
	\end{figure}
	\begin{proof}
		 First of all, we note that if we apply more half Dehn twists   to the given directions supported on the three $2$-punctured disks of each left diagram of Figure~\ref{A13} then we still have the same train track diagram having nonzero weights for the all arcs. Moreover, none of $x_{ij}$ or $x^{kl}$ is vanished. Then we can check that $N(SAT^3)$ is a $3$-bridge link which satisfies the rectangle condition. Actually, the index $3$ of $SAT^3$ makes the three rectangle types one by one. Now, assume that the index is greater than $3$. Then, especially, the rectangles carried by the green arcs of the left sides are preserved after the additional half Dehn twists as in Figure~\ref{A13}. Then, by using a similar argument we used to prove Theorem~\ref{T0} we complete the proof of this corollary.
		\end{proof}


\begin{thebibliography}{1}
   \bibitem{0} D. Bachman, S. Schleimer, {\em Distance and bridge position}, Pacific J. Math. 219(2005), no. 2, 221-235.
  \bibitem{1} A. Casson,  C. Gordon, {\em Manifolds with irreducible Heegaard splittings of arbitrary large genus}, Unpublished.

\bibitem{2} A. Coward, {\em Algorithmically detecting the bridge number of hyperbolic knots}, {\em preprinted}
\bibitem{3} J. Emert, C. Ernst, {\em $N$-string tangles}, J. Knot Theory Ramifications 9 (2000), no. 8, 987-1004.
  \bibitem{4} J. Hempel, {\em 3-manifolds as viewed from the curve complex}, Topology 40 (2001), no. 3, 631-657.
Princeton Univ. Press, 1974.
\bibitem{5} B. Kwon, {\em An algorithm to classify rational 3-tangles} J. Knot Theory Ramifications 24 (2015), no. 1, 1550004, 62 pp
 
 \bibitem{11} W. Menasco, {\em Closed incompressible surfaces in alternating knot and link complements}, Topology 23 (1984), no. 1, 37-44. 
  \bibitem{7} J. Otal, {\em Presentations en ponts des noeuds trivial}, C. R. Acad. Sci. Paris S\'er. I Math., 294:553-556, 1982
 \bibitem{8} M. Ozawa, K. Takao, {\em A locally minimal, but not globally minimal, bridge position of a knot}, Math. Proc. Cambridge Philos. Soc. 155 (2013), 181-190.
 \bibitem{9} R. Penner, J. Harer, {\em combinatorics of Train Tracks}, Annals of Mathematics Studies No. 125, Princeton University Press (1992), Sections 1.1-1.2.

  \bibitem{10} K. Takao, {\em Bridge decompositions with distances at least two}, Hiroshima Math. J. 42 (2012), 161-168.
\end{thebibliography}
\end{document}